\documentclass{amsart}
\usepackage{amssymb,mathrsfs}

\def\bC {\mathbf{C}}

\def\bR {\mathbf{R}}

\def\bZ {\mathbf{Z}}

\def\cH {\mathcal{H}}

\def\a {{\alpha}}

\newcommand{\Tr}{\operatorname{trace}}

\newcommand{\ba}{\begin{aligned}}
\newcommand{\ea}{\end{aligned}}

\newcommand{\be}{\begin{equation}}
\newcommand{\ee}{\end{equation}}

\newtheorem{Thm}{Theorem}[section]

\newtheorem{Prop}[Thm]{Proposition}
\newtheorem{Cor}[Thm]{Corollary}
\newtheorem{Lem}[Thm]{Lemma}
\newtheorem{Def}[Thm]{Definition}

\def\bea{\begin{eqnarray}}
\def\eea{\end{eqnarray}}

\usepackage{hyperref}
\hypersetup{pdfborder=0 0 0}

\usepackage{undertilde}
\usepackage{color}

\begin{document}

\title{
Husimi, Wigner, T\"oplitz,  
quantum
statistics and anticanonical transformations}

\author[T. Paul]{Thierry Paul}
\address{CNRS, 
LYSM – Laboratoire Ypatia de Sciences Mathématiques, Roma, Italia \& LJLL – Laboratoire Jacques-Louis Lions,  
Sorbonne Université 4 place Jussieu 75005 Paris France}
\email{thierry.paul@sorbonne-universite.fr}
\LARGE
\begin{abstract}
We study the behaviour of Husimi, Wigner and T\"oplitz symbols of quantum density matrices when quantum statistics are tested on them, that is when on exchange two coordinates in one of the two variables of their integral kernel. We show that to each of these actions is  associated a canonical transform on the cotangent bundle of the underlying classical phase space. Equivalently can one  associate  a complex canonical transform on the complexification of the phase-space. In the off-diagonal T\"oplitz representation introduced in \cite{tp3}, the action considered is associated to a complex aanticanonical relation.
\end{abstract}
\maketitle

\tableofcontents

\section{Introduction}
Quantum statistics is a fundamental hypothesis in quantum mechanics. In insures in particular the stability of matter. At the contrary of many other aspects of non-relativistic quantum mechanics which have a natural ```classical" counterpart, it seems difficult to associate to statistics properties of quantum object a classical corresponding symmetry. Changing the sign after permutation of coordinates of different particle doesn't appeal any classical simple action. Moreover most of the quantities which ``passes" at the limit of vanishing Planck constant are quadratic and therefore looks at insensible to the change of sign. Finally, typical fermionic expressions such as exchange term in the Hartree-Fock theory vanishes numerically at the limit $\hbar\to 0$.

In this little note, we will implement this ``exchange" action on three (in fact four) different symbols associated to  quantum density matrices: the Husimi function (average of the density matrix on coherent states, therefore a probability density), Wigner functions (tat is the Weyl suitably renormalized by a power of the Planck constant in order to be of integral $1$ (but non positive) and the  T\"oplitz symbol appearing in the so-called positive quantization procedure.

In these three symbolic situation, the result is that associated to the exchange action appears as the action of a complex or equivalently on a doubled space canonical transformation:

\begin{enumerate}
\item for the Husimi symbol (after a  weighting by a Gaussian weight), a direct action on the variables corresponding to a complex canonical transformation: the transform $\bar{z_i}\leftrightarrow \bar{z_j}$ $z_i, z_j$ remaining unchanged. the complex canonical transform is of the form $\begin{pmatrix}
0&i\\i&o
\end{pmatrix}$.
\item idem for the T\"oplitz symbol, with a different  Gaussian weight
\item for thw Wigner symbol (renormalized Weyl symbol), the above-mentioned complex transform is seen as a canonical transformtion on the cotangent bundle of the phase space. This transformation is the composition of permutation of variables and a ``Fourier rotation" $q_i\to p_i,\ p_i\to -q_i$ and the exchange acts on the Wigner function  by the metaplectic (in a doubled dimension space) representation, namely exchange of coordinates plus Fourier transform. In particular it doesn't act by a metapletic type representation of the complex linear symplectic group.
\vskip 0.3cm
\noindent To get such a feature, one has to go the off-diagonal T\"oplitz caluclus introduced in \cite{tp3} and is this time associated to a an anticanonical transformation, that is a transformation which maps the sympletic form to its opposite.
\item the off-diagonal symbol is mapped by the action of the metaplectic representation of the anticanonical linear transformation $\begin{pmatrix}
0&i\\-i&0
\end{pmatrix}$. See Sections \ref{offtop} and mostly \ref{linkmeta} for details.
\end{enumerate}

The conclusion to which  all this (sometimes only formal) computations lead is the fact that, at a ``classical" level, quantum statistics involve transformation which don't preserve the usual symplectic cotangent bundle of the configuration space: either one has to  pass in a non trivial way to the cotangent bundle of the cotangent bundle itself, either one has to non preserve the sympletic structure, and allow anticanonical transformations. 
\vskip 1cm
\textbf{The underlying classical picture of bosons and fermions  either lives on the cotangent space of the classical phase space, {\begin{color}{red} i.e. in the setting of second quantization,\end{color}}  or involves antisymplectic symmetries.}
\section{Quantum statistics}
On the setting of indistinguishable quantum particles, a state is a density matrix, i.e. a positive trace one operator on $\cH^{\otimes N}$, invariant by permutations of the factors in the tensorial product. we have denoted $\cH=L^2(\bR^d)$.
\begin{Def}
 Let $\rho$ be a density matrix given by an integral kernel $\rho(X;Y),\ X=(x_1,\dots,x_n),\ Y=(y_1,\dots,y_n)$. We define, for $i,j=1,\dots,N$, the mappings 
 
$U_{i\leftrightarrow j}:\ \rho(X;Y)\to U_{i\to j}\rho(X;Y)=\rho(X;Y)|_{y_i\leftrightarrow y_j}$ 

and 

$V_{i\leftrightarrow j}:\ \rho(X;Y)\to V_{i\to j}\rho(X;Y)=\rho(X;Y)|_{x_i\leftrightarrow x_j}$.
\end{Def}
In terms of density matrices, quantum statistics will be seen as looking at density matrices which are eigenvectors of eigenvalue $1$ or $-1$ of the two mappings  $U_{i\leftrightarrow j},V_{i\leftrightarrow j}$.

The indistinguishability property of the quantum system reads as
\be\label{indis}
U_{i\leftrightarrow j}V_{i\leftrightarrow j}
=
V_{i\leftrightarrow j}U_{i\leftrightarrow j},\ \ \ \forall i,j=1,\dots,N.
\ee

\section{Husimi}\label{hus}
Let us recall that the Husimi function of a density matrix $\rho$ is defined as
\be\label{dejhus}
\widetilde W[{\rho}](Z,\bar Z)
=
\frac1{(2\pi\hbar)^{dN}}
\langle\varphi_Z|\rho|\varphi_Z\rangle,
\ee
where, for $Z=q+ip\in\bZ^{dN}$ and $x\in\bR^{dN}$, 
\be\label{defec}
\varphi_Z(x)
=
\frac1{(\pi\hbar)^\frac{dN}4}
e^{-\frac{(x-q)^2}{2\hbar}}e^{i\frac{p.x}\hbar}.
\ee
The most elementary properties of the Husimi transform are
\be\label{elemprophus}
\widetilde W[{\rho}]\geq 0
\mbox{ and }\int_{\bZ^{dN}}\widetilde W[{\rho}](Z)dZ=\Tr\rho=1.
\ee
Our first link between quantum statistics and the classical underlying space is the contents of the following result.
\begin{Lem}
%

Let us consider the Husimi function of $\rho$, $\widetilde W[\rho](Z,\bar Z)$ expressed on the complex variables $Z=(z_1,\dots,z_n)$, $z_l=~q_l+ip_l,\ \bar z_l=q_l-ip_l$.

Then
\[
\widetilde W[{U_{i\leftrightarrow j}\rho}](Z,\bar Z)
=
e^{-\frac{(\bar z_i-\bar z_j)(z_i-z_j)}{2\hbar}}\widetilde W[\rho](Z,\bar Z)|_{z_i\leftrightarrow z_j}
\]
\[
\widetilde W[{V_{i\leftrightarrow j}\rho}](Z,\bar Z)
=
e^{-\frac{|z_i-z_j|^2}{2\hbar}}
\widetilde W[\rho](Z,\bar Z)|_{\bar z_i\leftrightarrow \bar z_j}
\]
\end{Lem}
\vskip 1cm
Note that, as expected,
\[
\widetilde W[{V_{i\leftrightarrow j}U_{i\leftrightarrow j}\rho}](Z,\bar Z)
=
\widetilde W[\rho](z,\bar z)|_{z_i\leftrightarrow z_j,\ \bar z_i\leftrightarrow \bar z_j}
\]
\vskip 1cm
Note also that, with te definition
\be\label{defzpm}
z_\pm=q_\pm+ip_\pm:=\frac{z_i\pm z_j}{\sqrt 2},
\ee
\be\label{eqhus}\nonumber
\begin{pmatrix}
z_i\\z_j\\\bar z_i\\\bar z_j
\end{pmatrix}
\to
\begin{pmatrix}
z_j\\z_i\\\bar z_i\\\bar z_j
\end{pmatrix}   
\Longleftrightarrow\   
\begin{pmatrix}
z_+\\z_-\\\bar z_+\\\bar z_-
\end{pmatrix}
\to
\begin{pmatrix}
z_+\\-z_-\\\bar z_+\\\bar z_-
\end{pmatrix}
\Longleftrightarrow  
\begin{pmatrix}
q_+\\q_-\\p_+\\p_-
\end{pmatrix}
\to
\begin{pmatrix}
q_+\\-ip_-\\p_+\\ iq_-
\end{pmatrix}
\ee
so the complex metaplectic transform associated to the exchange term is the matrix $I_+\otimes S^c_H$ with
\be\label{cSHus}
S^c_H=\begin{pmatrix}
0&-i\\i&0
\end{pmatrix},\ \det{S^c_H}=
-
1.
\ee
\begin{Cor}
A density matrix $\rho$ is bosonic if and only if,  for all $i,j=1,\dots,n$,
\begin{eqnarray}
\widetilde W[\rho](Z,\bar Z)|
&=&
e^{-\frac{(\bar z_i-\bar z_j)(z_i-z_j)}{2\hbar}}
\widetilde W[\rho](Z,\bar Z)|_{ z_i\leftrightarrow  z_j}\nonumber\\
&=&
e^{-\frac{(\bar z_i-\bar z_j)(z_i-z_j)}{2\hbar}}
\widetilde W[\rho](Z,\bar Z)|_{\bar z_i\leftrightarrow \bar z_j}.\nonumber
\end{eqnarray}
\end{Cor}

\begin{Cor}
Let $n=2$. A density matrix $\rho$ is bosonic if and only if
\begin{eqnarray}
\widetilde W[\rho](Z,\bar Z)|
&=&
e^{\frac{(\bar z_1-\bar z_2)(z_1-z_2)}{4\hbar}}
H(z_1-z_2,\bar z_1-\bar z_2,z_1+z_2,\bar z_1+\bar z_2)\nonumber
\end{eqnarray}
with $H$ even (separately) in the two first variables.
\end{Cor}
\section{Wigner}\label{wig}
The Wigner function of a density matrix is nothing but its Weyl symbol, divided by $(2\pi\hbar)^{dN}$. More precisely the Wigner function of $\rho$ is defined as 
\be\label{defwig}
W[\rho](X,\Xi)=\int_{\bR^{2dN}}
\rho(X+\hbar\frac\delta2,X-\hbar\frac\delta2)e^{i\frac{X.\Xi}\hbar}d\delta
\ee
At the contrary of the Husimi function, $W[\rho]$ is not positive, but its main elementary properties are
\begin{eqnarray}\label{elempropwig}
\int_{\bR^{2dN}}W[\rho](X,\Xi)dXd\xi=\Tr\rho&=&1\\
\mbox{ and }&&\nonumber\\
\frac1{(2\pi\hbar)^{dN}}\int_{\bR^{2dN}}W[\rho](X,\Xi0W[\rho'](X,\Xi)dXd\Xi&=&\Tr{(\rho\rho')}.\nonumber
\end{eqnarray}

Let us now define the semiclassical 
 symplectic Fourier transform as 
\[
f({\widehat{q,p}}^\hbar)
=
\frac1{(2\pi\hbar)^d}
\int_{\bR^d\times\bR^d} f(x,\xi)e^{i\frac{q\xi-px}\hbar}dxd\xi.
\]
Note that, at the difference of the usual Fourier transform:
\[
f(\widehat{\widehat{x,\xi}^\hbar}^\hbar)=f(x,\xi)
\]
Let $a_\mp=\frac{a_i\mp a_j}{\sqrt2}$ for $a=q,p,y,\xi$. And let omit the dependence in the variable $q_1,\dots,q_{i-1},q_{i+1},\dots,q_{j-1},q_{j+1},\dots, q_N$ and the same for $p$.

We denote
\[
W^{\frac\pi2}[\rho](x_+,\xi_+;x_-,\xi_-)=W[\rho](x_i,x_j;\xi_i,\xi_j).
\]\begin{Lem}
\[
W^{\frac\pi2}[U_{i\leftrightarrow j}\rho](q_+,p_+;p_-,q_-)=
W^{\frac\pi2}[\rho](q_+,p_+;\widehat{q_-,p_-}^\hbar)
\]
\[
W^{\frac\pi2}[V_{i\leftrightarrow j}\rho](q_+,p_+;p_-,q_-)=
W^{\frac\pi2}[\rho](q_+,p_+;\widehat{-q_-,-p_-}^\hbar)
\]
\end{Lem}
Note that
\[
W[V_{i\leftrightarrow j}U_{i\leftrightarrow j}\rho](q_1,p_1,\dots,q_i,p_i,\dots,q_j,p_j,\dots,q_n,p_n)
=\]
\[
W[\rho]
(q_1,p_1,\dots,q_{i-1},p_{i-1},q_j,p_j,\dots,q_{j-1},p_{j-1},
q_i,p_i,\dots,q_n,p_n)
\]
\begin{proof}
It is enough to isolate the $ij$ block.

$U_{i\leftrightarrow j}\rho(x_i,x_j;y_1,y_j)=\rho(x_i,x_j:y_j,y_i)$.  So
\large
\bea
&&(2\pi\hbar)^{2d}W[U_{i\leftrightarrow j}\rho](q_i,q_j;p_i,p_j)\nonumber\\
&=&
\int d\delta_id\delta_j\rho(q_i+\delta_i,q_j+\delta_j;q_j-\delta_j,q_i-\delta_i)e^{-2ip\cdot\delta/\hbar}\nonumber\\
&=&\int
d\delta d\eta W[\rho]\nonumber\\
&&
((q_i+q_j+\delta_i-\delta_j)/2,(q_j+q_i+\delta_j-\delta_i)/2;\eta)e^{\frac i\hbar(q_i-q_j+\delta_i+\delta_j)\eta_i+\eta_j(q_j-q_i+\delta_i+\delta_j)}\nonumber\\
&&e^{-2ip\cdot\delta/\hbar}\ (=e^{-2ip_1(\delta/\hbar})\nonumber\\
&=&\nonumber
\int d\delta d\eta W[\rho](q_i+q_j+\delta,q_i+q_j-\delta,\eta)
\delta(\eta_i+\eta_j-(p_i+p_j))
e^{-2i(p_i-p_j)\delta/\hbar}\\
&=&
\int dy d\eta 
\delta(\eta_i+\eta_j-(p_i+p_j))
\delta(y_i+y_j-(q_i+q_j))
\nonumber\\
&&
e^{i((q_i-q_j)(\eta_i-\eta_j)-(p_i-p_j)(y_i-y_j))/\hbar}
W[\rho](y;\eta)\nonumber
\eea
\LARGE
Let us perform the change of variable $a_\mp=\frac{a_i\mp a_j}{\sqrt2}$ for $a=q,p,y,\eta$. This correspond to the metaplectic mapping:
\[
R(\frac\pi2)=
\scriptsize
\begin{pmatrix}{\begin{pmatrix}\frac1{\sqrt2}&-\frac1{\sqrt2}\\\frac1{\sqrt2}&\frac1{\sqrt2}
\end{pmatrix}}&0&0&0\\
0&{\begin{pmatrix}\frac1{\sqrt2}&-\frac1{\sqrt2}\\\frac1{\sqrt2}&\frac1{\sqrt2}
\end{pmatrix}}&0&0\\
0&0&{\begin{pmatrix}\frac1{\sqrt2}&-\frac1{\sqrt2}\\\frac1{\sqrt2}&\frac1{\sqrt2}
\end{pmatrix}}&0\\
0&0&0&{\begin{pmatrix}\frac1{\sqrt2}&-\frac1{\sqrt2}\\-\frac1{\sqrt2}&\frac1{\sqrt2}
\end{pmatrix}}
\end{pmatrix}
\]
on
\[
\begin{pmatrix}
q_i\\q_j\\\xi_i\\\xi_j\\p_i\\p_j\\x_i\\x_j
\end{pmatrix}
\in T^*(T^*\bR^d, dq\wedge d\xi+dp\wedge dx).
\]

 Note that both $$dq\wedge dp=dq_+\wedge dp_++dq_-\wedge dp_-=d\tilde q\wedge d\tilde p$$ and
 $$dq\wedge d\xi+dp\wedge dx=d\tilde q\wedge d\tilde \xi+d\tilde p\wedge d\tilde x$$ where $\tilde a=\begin{pmatrix}a_+\\a_-\end{pmatrix}$.
 
We denote $W^{\frac\pi2}[\rho](y_+,\eta_+;\eta_-,y_-)$ and
$W^{\frac\pi2}[U_{i\leftrightarrow j}\rho](q_+,p_+;p_-,q_-)$.
We get
\[
W^{\frac\pi2}[U_{i\leftrightarrow j}\rho](q_+,p_+;p_-,q_-)=
W^{\frac\pi2}[\rho](q_+,p_+;\widehat{q_-,p_-}^\hbar)
\]
\end{proof}
Let us call now $W^-$ the Wigner function (done with the symplectic Fourier transform) on the two variables $q_-,p_-$, namely, 
\large
$$\small W^-\big[W^{\frac\pi2}[\rho]\big](q_+,p_+|p_-,q_-;x_-,\xi_-)=$$
$$\small\int\overline{W^{\frac\pi2}[\rho]\big](q_+,p_+,p_-+2\delta\hbar,q_-+2\delta'\hbar)}
$$
$$
W^{\frac\pi2}[\rho]\big](q_+,p_+,p_--2\delta\hbar,q_--2\delta'\hbar)e^{i(x_-\delta-\xi_-\delta')}
d\delta d\delta'.$$
\LARGE
\vskip 0.3cm
$ W^-\big[W^{\frac\pi2}[\rho]\big]$ 
 lives on $T^*(\bR^d_{q_+})\times T^*(\bR^{2d}_{(p_-,q_-)})$
equiped with the symplectic form
$$
dq_+\wedge qp_+
+dq_-\wedge d\xi_-+dp_-\wedge dx_-.
$$

One has
\[
W^-\big[W^{\frac\pi2}[U_{i\leftrightarrow j}\rho]\big](q_+,p_+|p_-,q_-;x_-,\xi_-)=
\]
\[
W^-\big[W^{\frac\pi2}[\rho]]
(q_+,p_+|-\xi_-,-x_-;q_-,p_-)
\]
That is, the action of $U_{i\leftrightarrow j}$ on $\rho$ is seen on $ W^-\big[W^{\frac\pi2}[\rho]\big]$ by the pointwise action of the following matrix:
\[
S=
\begin{pmatrix}
S_+&0\\0&S_-
\end{pmatrix}
=
\begin{pmatrix}
\begin{pmatrix}1&0&0&0\\
0&1&0&0\\
0&0&1&0\\
0&0&0&1\\
\end{pmatrix}&0\\
0&
\begin{pmatrix}
0&0&0&1\\
0&0&-1&0\\
0&1&0&0\\
-1&0&0&0\\
\end{pmatrix}
\end{pmatrix}
\mbox { on }
\begin{pmatrix}
q_+\\ \xi_+\\ p_+\\ x_+\\ q_-\\ \xi_-\\ p_-\\ x_-
\end{pmatrix}
\]

and this matrix is symplectic.

Defining now $z_\pm=p_\pm+ix_\pm,\ \theta_\pm=q_\pm+i\xi_\pm$ we find that $S$ becomes on these new variables,
$S^c=(S^c_+,S^c_-)=(I,i\scriptsize{\begin{pmatrix}
0&1\\1&0
\end{pmatrix}}
)
$
And so the complex  metapletic transform  associated is 
\[
S^c_W=\begin{pmatrix}
0&i\\i&o
\end{pmatrix},\ \det{S^C_W}=1.
\]

\section{T\"oplitz}\label{top}
Let $\rho$ be a  T\"oplitz operator of symbol.
$
\utilde W[\rho].
$
This means that $\rho$ can ve written as
\be\label{deftop}
\rho
=
\frac1{(2\pi\hbar)^{dN}}
\int_{\bC^{dN}}
\utilde W[\rho](Z,\bar Z)|\varphi_Z\rangle\langle\varphi_Z|dZ
\ee
(here the integral as to be understood in the weak sense on $\cH$).
Elementary properties of $\utilde W[\rho]$ are
\be\label{elemproptop}
\utilde W[\rho]\geq 0\Rightarrow\rho>0,\mbox{ and }\int_{\bC^{dN}}\utilde W[\rho]dZ=\Tr\rho.
\ee
Moreover,the secondd property of \eqref{elempropwig} can be ``disintegrated" in the following couplig between Husimi and T\"oplitz settings:
\be\label{elemhustop}
\int_{\bC^{dN}}\widetilde W[\rho](Z,\bar Z)\utilde W[\rho'](Z,\bar Z)dZ
=
\Tr{(\rho\rho')}.
\ee

\begin{Lem}
\large
\[
\utilde W[U_{i\leftrightarrow j}\rho](z_i,\bar z_i,z_j,\bar z_j)=
e^{-\frac{|z_i-z_j|^2}{2\hbar}}
\utilde W[\rho](z_j,\bar z_i,z_i,\bar z_j)
\]
\[
\utilde W[U_{i\leftrightarrow j}\rho](q_-,p_-;q_+,p_+)
=
e^{-\frac{q_-^2+p_-^2}{2\hbar}}
\utilde W[\rho](-ip_-,iq_-;q_+,p_+)
\]
\vskip 0.5cm
\[
\utilde W[V_{i\leftrightarrow j}\rho](q_1,p_1,\dots,q_i,p_i,\dots,q_j,p_j,\dots,q_n,p_n)
=
e^{-\frac{(q_i-q_j)^2+(p_i-p_j)^2}{2\hbar}}
\]
\[\times
\utilde W[\rho]
(q_1,p_1,\dots,q_{i-1},p_{i-1},-ip_j,iq_j,\dots,q_{j-1},p_{j-1},
-ip_i,iq_i,\dots,q_n,p_n)
\]
\[
=
e^{-\frac{(q_i-q_j)^2+(p_i-p_j)^2}{2\hbar}}
\utilde W[\rho]|_{\substack{\ \\z_i\leftrightarrow -z_j\\\bar{z_i}\leftrightarrow\bar{z_j}}},\ z_i=q_i+ip_i.
\]

\end{Lem} 

In other words, the exchange action on the T\"oplitz symbol is the same as the one on the Husimi function, modulo a different gaussian weight.
\section{On Wigner again}\label{ow}
Let us denote
\[
 U_{i\leftrightarrow j}^WW[\rho]=W[U_{i\leftrightarrow j}\rho]
 \]
 Let us moreover denote by $W^2[\rho]$ the Wigner function of the Wigner function of $\rho$ (see footnote 1):
 \[
 W^2[\rho]=W[W[\rho]].
 \]
 Let us denote by $Q_i=(q_i,\xi_i)$ and $P_i=(p_i,x_i),\ i=1,\dots, N,$ the variables in $T^*(T^*\bR^d))$. We define:
 \[
 Q_i^t=(\xi_i,q_i),\ \ \ P_i^t=(x_i,p_i).
 \]
 
 \begin{Lem}
 \large
\[
 W^2[U_{i\leftrightarrow j}\rho](Q_1,P_1,\dots,Q_i,P_i,\dots,Q_j,P_j,\dots,Q_n,P_n)
=\]
\[
 W^2[\rho]
(Q_1,P_1,\dots,Q_{i-1},P_{i-1},P^t_j,-Q^t_j,\dots,Q_{j-1},P_{j-1},
P^t_i,-Q^t_i,\dots,Q_n,P_n)
\]
\[
 W^2[U_{i\leftrightarrow j}\rho]=W[U^W_{i\leftrightarrow j}W[\rho]]=
W^2[\rho]|_{\substack{\ \\Q_i\leftrightarrow P^t_j\\P_i\leftrightarrow-Q^t_j}}.
\]
\vskip 1cm
\[
 W^2[V_{i\leftrightarrow j}\rho](Q_1,P_1,\dots,Q_i,P_i,\dots,Q_j,P_j,\dots,Q_n,P_n)
=\]
\[
 W^2[\rho]
(Q_1,P_1,\dots,Q_{i-1},P_{i-1},-P^t_j,Q^t_j,\dots,Q_{j-1},P_{j-1},
-P^t_i,Q^t_i,\dots,Q_n,P_n)
\]
\[
 W^2[V_{i\leftrightarrow j}\rho]=W[V^W_{i\leftrightarrow j}W[\rho]]=
W^2[\rho]|_{\substack{\ \\Q_i\leftrightarrow -P^t_j\\P_i\leftrightarrow Q^t_j}}.
\]
%
\end{Lem}
So $U^W_{i\leftrightarrow j},\ V_{i\leftrightarrow j}$ are metaplectic operators  associated to canonical transforms on $T^*(T^*(\bR^{dN}))$.
\begin{Lem}
Denoting now $z_i=q_i+\xi_i,\ \theta_i=p_i+ix_i$ we have
\[
W[U^W_{i\leftrightarrow j}W[\rho]]=W^2[U_{i\leftrightarrow j}\rho]=W^2[\rho]|_{\substack{z_i\leftrightarrow iz_j\\
\theta_i\leftrightarrow i\theta_j}}
\]
\[
W^2[V_{i\leftrightarrow j}\rho]=W^2[\rho]|_{\substack{z_i\leftrightarrow -iz_j
\\
\theta_i\leftrightarrow -i\theta_j}}
\]
\end{Lem}
So $U^W_{i\leftrightarrow j},\ V_{i\leftrightarrow j}$ are metaplectic operators  associated to complex canonical transforms on the complexification of $T^*(\bR^{dN})$.

\section{Off-diagonal T\"oplitz representations}\label{offtop}

In this section, we take $d=1$ and $N=2$. 

A density matrix $\rho$ has an integral kernel $\rho(x_1,x_2;y_1,y_2)$ and 
$$
(U\rho)(x_1,x_2;y_1,y_2)=\rho(x_1,x_2;y_2,y_1)
$$
$$
(V\rho)(x_1,x_2;y_1,y_2)=\rho(x_2,x_1;y_1,y_2).
$$
therefore, performing a change of variables 
$$x=(x_1-x_2)/\sqrt2, x' =(x_1+x_2)/\sqrt2,$$
$$
y=(y_1-y_2)/\sqrt2,y'=(y_1-y_2)/\sqrt2,
$$
one get, with a slight abuse of notation that
$$
U\rho(x,y;x',y')=\rho(x,-y:x',y')
$$
$$
V\rho(x,y;x',y')=\rho(-x,y:x',y')
$$ 
In the rest of this section we will omit the variables $x',y'$.

Let us consider a  T\"oplitz operator
\[
H=\int h(z)\vert\psi_z\rangle\langle\psi _z\vert\frac{dzd\bar z}{2\pi\hbar},
\]
where, for $z=q+ip$,
$$
\psi _{z}=
\frac{e^{-\frac{(x-q)^2}{2\hbar}}e^{i\frac{p x}\hbar}}{(\pi\hbar)^\frac14}.$$

Following Bargamnn's philosophy, we remark that, for each $z,z'$,
$$
\langle x\vert\psi_z\rangle\langle\psi _{z'}\vert y\rangle
=
e^{-\frac{\bar z^2}{4\hbar}}e^{-\frac{x^2-2\bar zx}{2\hbar}}e^{-\frac{\bar zz}{4\hbar}}e^{-\frac{\bar z'z'}{4\hbar}}e^{-\frac{y^2-2 z'y}{2\hbar}}e^{-\frac{ {z'}^2}{4\hbar}}
$$
Therefore
\begin{eqnarray}
\langle x\vert\psi_z\rangle\langle\psi _z\vert -y\rangle&
=&
e^{-\frac{z\bar z}{\hbar}}
\langle x\vert\psi_z\rangle\langle\psi _z\vert y\rangle
\nonumber\\
\langle -x\vert\psi_z\rangle\langle\psi _z\vert y\rangle
&
=&
e^{-\frac{z\bar z}{\hbar}}
\langle x\vert\psi_z\rangle\langle\psi _z\vert y\rangle
\nonumber
\end{eqnarray}
Let us define $H^l$ by its integral kernel $H^l(x,y)=H(-x,y)$ where $H(x,y)$ is the integral kernel of $H$. Let $H^r$ be defined the same way by $H^r(x,y)=h(x,-y)$.

Obviously
\newcommand{\lr}{{\substack{l\\r}}}
\[
H^\lr=\int h(z)\vert\psi _{\mp z}\rangle\langle\psi _{\pm z}\vert\frac{dzd\bar z}{2\pi\hbar}.
\]
%
%
%
%
%
Therefore, we get the following off-diagonal expressions.
\begin{Lem}\label{offdiag}
\bea
VH&=&\int h(q,p)
\vert\psi_{-z}\rangle\langle\psi_z\vert\frac{dzd\bar z}{2\pi\hbar}\nonumber\\
UH&=&\int h(q,p)
\vert\psi_{z}\rangle\langle\psi_{-z}\vert\frac{dzd\bar z}{2\pi\hbar}\nonumber\\
UVH&=&\int h(q,p)
\vert\psi_{-z}\rangle\langle\psi_{-z}\vert\frac{dzd\bar z}{2\pi\hbar}\nonumber\\
U^2=V^2&=&1\nonumber
\eea
\end{Lem}
These expressions have to be compared to the following ones, derived form Section \ref{top}.
\begin{Lem}
\bea
VH&=&\int h(ip,-iq)e^{-\frac{q^2+p^2}{2\hbar}}
\vert\psi_{z}\rangle\langle\psi_z\vert\frac{dzd\bar z}{2\pi\hbar}\nonumber\\
UH&=&\int h(-ip,iq)e^{-\frac{q^2+p^2}{2\hbar}}
\vert\psi_{z}\rangle\langle\psi_{z}\vert\frac{dzd\bar z}{2\pi\hbar}\nonumber
\\
UVH&=&\int h(-q,-p)
\vert\psi_{z}\rangle\langle\psi_{z}\vert\frac{dzd\bar z}{2\pi\hbar}\nonumber
\eea
\end{Lem}
The T\"oplitz symbol of $VH$ (resp. $UH$) is $h_V(q,p)=h(ip,-iq)e^{\frac{q^2+p^2}{2\hbar}}$ (resp.  $h_U(q,p)=h(-ip,iq)e^{\frac{q^2+p^2}{2\hbar}}$).
\begin{Lem}
\noindent Let $h\geq 0, \int h=1$. 

Then $H^B:=\tfrac14(H+VH+UH+UVH)$ is a bosonic state,

\  and $H^F:=\tfrac14(H-VH-UH+UVH)$ is a fernionic one.
\end{Lem}
\begin{proof}
One has $H^B=VH^B=UH^B=UVH^B$, $\mbox{Tr }H^B=1$,
 $H^F=-VH^B=-UH^B=UVH^B$, $\mbox{Tr }H^B=1$, and 
\[
H^B=\tfrac14\int h(q,p)
\vert\psi_{z}+\psi_{-z}\rangle\langle\psi_z+\psi_{-z}\vert\frac{dzd\bar z}{2\pi\hbar}\geq 0.
\]
\[
H^F=\tfrac14\int h(q,p)
\vert\psi_{z}-\psi_{-z}\rangle\langle\psi_z-\psi_{-z}\vert\frac{dzd\bar z}{2\pi\hbar}\geq 0.
\]
\end{proof}
Finally, $H^B$ is ``semiclassical".

\section{Link with the complex (anti)metaplectic representation}\label{linkmeta}
\newcommand{\ccomp}[1]{\mathcal C(#1)}
We have seen in the previous (sub)sections that $U$ (resp. $V$) is associated to the action of the matrix $\begin{pmatrix}
0&-i\\i&0
\end{pmatrix}$ (resp. $\begin{pmatrix}
0&i\\-i&0
\end{pmatrix}$) on the Husimi function and the Toeplitz symbol. Therefore it is natural to think that $U$ (resp. $V$) should be associated to the ``metaplectic" quantization of $\begin{pmatrix}
0&-i\\i&0
\end{pmatrix}^{-1}=\begin{pmatrix}
0&i\\-i&0
\end{pmatrix}$ (resp. $\begin{pmatrix}
0&i\\-i&0
\end{pmatrix}^{-1}=\begin{pmatrix}
0&-i\\i&0
\end{pmatrix}$),  ``metaplectic" because these matrices are not canonical. Precisely, a definition of quantization of anticanonical mappings has been provide in the preceding section that we can use in the present situation.

With the 
%
%
%
%
%
%
%
definition of $\ccomp{S}$  in  \cite
[Definition 10, Section 7]
{tp3},
we get  our final result, as a direct application of \cite
[formula (7.1)]
{tp3},.
\begin{Lem}
Let $H$ a Toeplitz operator of symbol $h(q,p)$. Then
\bea
UH&=&
\ccomp{\scriptsize\begin{pmatrix}
0&i\\-i&0
\end{pmatrix}}H.\nonumber\\
VH&=&
\ccomp{\scriptsize\begin{pmatrix}
0&-i\\i&0
\end{pmatrix}}H
\nonumber
\eea

\end{Lem}

But the ``true" result is the following, that we express only  for $U$, the case $V$ being straightforwardly the same).
\begin{Prop}\label{propnoncan}
$$
UH=T^{off}\left[\sigma^{off}[\ccomp{\scriptsize\begin{pmatrix}
0&i\\-i&0
\end{pmatrix}}H]
\right],
$$
where $\sigma^{off}[\ccomp{\scriptsize\begin{pmatrix}
0&i\\-i&0
\end{pmatrix}}H]$ is defined by \cite
[Section 6.2]
{tp3}, and  $T^{off}$ by the off-diagonal Toeplitz quantization formula \cite
[formula (6.1)]
{tp3},.

Namely,  $UH$ is given by the off-diagonal Toeplitz quantization of the off-diagonal Toeplitz symbol of $\ccomp{\scriptsize\begin{pmatrix}
0&i\\-i&0
\end{pmatrix}}H$  {\bf without} the multiplication
 by the factor $e^{-\frac{|\cdot|^2}{\hbar}}$ 
 as for the Husimi and the (diagonal) Toeplitz cases, as seen in the previous sections.

\end{Prop}
Proposition \ref{propnoncan} shows clearly first that the exchange mappings $U,V$ are clearly associated to complex non-canonical linear transformations, and second that the off-diagonal Toeplitz quantization/representation of $M^\pm(2,\bC):=\{S\in SL(2,\bC),\ \det S=\pm 1\}$ established in \cite[Section 7]{tp3}, is meaningful.

Note again that $\mbox{det }{\small\begin{pmatrix}
0&i\\-i&0
\end{pmatrix}}=\mbox{det }\small\begin{pmatrix}
0&-i\\i&0
\end{pmatrix}=-1$.

\section{A classical phase space with symmetries inherited form quantum statistics}\label{clasphasstat}
%
%
%
The construction of the preceding (sub)section suggests that a noncommutative  extension of the usual phase space of classical mechanics, namely the cotangent bundle of the configuration space, is possible in order to handle the trace, at the classical underlying level, of more  symmetries, coming from the quantum one, than the one usually considered: namely the unitary in a Hilbert space of the quantum propagation leading to the symplectic classical evolution associated to $SL(2,\bR)$. One recovers  the presence of the fundamental (as responsible, e.g., of the stability of matter) spin-statistics symmetries at the classical level by extending the group of symmetry  $SL(2,\bR)$ acting pointwise on the phase-space to $M^\pm(2,\bC)=\{S\in M(2,\bC), \det{S}=\pm 1\}$ with its action on a noncommutative space  established in 
\cite{tp3}.

\end{document}